\documentclass{amsart}

\usepackage{enumerate}
\usepackage{ifpdf}
\usepackage{amsmath}
\usepackage{amsfonts}
\usepackage{amssymb}
\usepackage{amsthm}
\usepackage[hyperfootnotes=false]{hyperref}
\usepackage{setspace}
\usepackage{graphicx}

\newtheorem{thm}{Theorem}[section]

\newtheorem{cor}[thm]{Corollary}

\newtheorem{lemma}[thm]{Lemma}

\theoremstyle{definition}
\newtheorem{defn}[thm]{Definition}

\theoremstyle{remark}
\newtheorem{remark}[thm]{Remark}

\newcommand{\la}{\lambda}
\newcommand{\vp}{\varphi}
\newcommand{\CC}{\mathbb C}

\date{\today}

\ifpdf
\author{Han Peters}
\address{KdV Institute for Mathematics\\
University of Amsterdam\\
The Netherlands}
\email{h.peters@uva.nl}

\author{Liz Raquel Vivas}
\address{Universidad Federal Fluminense\\
Rio de Janeiro\\
Brasil}
\email{lizvivas@impa.br}

\title[Vertical disks of polynomial skew product]{Polynomial skew-products with vertical Fatou-disk converging to the Julia set}

\begin{document}
\bibliographystyle{plain}

\begin{abstract}
Little is known about the existence of wandering Fatou components for rational maps in two complex variables. In 2003 Lilov proved the non-existence of wandering Fatou
components for polynomial skew-products in the neighborhood of an invariant super-attracting fiber. In fact Lilov proved a stronger result, namely that the forward orbit of
any vertical disk must intersect a fattened Fatou component of the invariant fiber.

Naturally the next class of maps to study are polynomial skew-products with an invariant attracting (but not super-attracting) fiber. Here we show that Lilov's stronger
result does not hold in this setting: for some skew-products there are vertical disks whose orbits accumulate at repelling fixed points in the invariant fiber, and that
therefore never intersect the fattened Fatou components. These disks are necessarily Fatou disks, but we also prove that the vertical disks we construct lie entirely in the
Julia set.

Our results therefore do not answer the existence question of wandering Fatou components in the attracting setting, but show that the question is considerably more
complicated than in the super-attracting setting.
\end{abstract}

\maketitle

\tableofcontents


\section{Introduction}\label{section:intro}

In this article we investigate the existence of wandering Fatou components for polynomial maps in two complex variables. A
natural place to investigate the existence problem is the class of polynomial skew-products, as these maps lie on the boundary between one and two-dimensional dynamics.

A skew-product is a map of the form
$$
F(t,z) = (f(t), g(t,z)) = (f(t), g_t(z)).
$$
If $f(t_0) = t_0$ then the fiber $\{t=t_0\}$ is mapped to itself by the function $g_{t_0}$. Hence if $g_{t_0}(z)$ is a polynomial, then there are no $1$-dimensional wandering
Fatou components in this fiber.

Consider the case where $|f^\prime(t_0)|<1$, so that all nearby fibers are attracted to the $t_0$-fiber. In this case each one-dimensional Fatou components of $g_{t_0}$ in the
$t_0$-fiber is contained in a $2$-dimensional Fatou component of $F$. One naturally wonders if all nearby Fatou components are eventually mapped onto one of these fattened
(pre-)periodic Fatou components. Indeed, it was shown in the PhD thesis of Lilov \cite{Lilov} that this is the case, under the stronger assumption that the $t_0$-fiber is
super-attracting, i.e. that $f^\prime(t_0) = 0$. Here we investigate whether his results also hold in the geometrically attracting case, i.e. when $0 < |f^\prime(t_0)| < 1$.

Lilov actually proved a much stronger result which implies the non-existence of wandering Fatou components. We denote by $\mathcal{A} \subset \mathbb C_t$ the immediate basin
of the superattracting fixed point $t_0$.

\begin{thm}[Lilov,2003]\label{thm:lilov}
Let $t_1 \in \mathcal{A}$, and let $D$ be an open one-dimensional disk lying in the $t_1$-fiber. Then the forward orbit of $D$ must intersect one of the fattened Fatou
components of $g_{t_0}$.
\end{thm}

An immediate corollary is that there are no wandering Fatou components in the set $\mathcal{A}\times\mathbb C_z$.

In this paper we prove that Theorem \ref{thm:lilov} does not hold in the geometrically attracting case. More precisely, we study explicit maps of the form
$$
F(t,z) = (\alpha t, p(z) + q(t)),
$$
where $\alpha < 1$ and $p$ and $q$ are polynomials.

\begin{thm}\label{thm:main}
There exist a triple $(\alpha, p, q)$ and a vertical holomorphic disk $D \subset \{t = t_0\}$ whose forward orbit accumulates at a point $(0,z_0)$, where $z_0$ is a repelling
fixed point in the Julia set of $p$.
\end{thm}

In particular it follows that the forward orbits of $D$ will never intersect the fattened Fatou components of $p$.

The disks $D$ will be Fatou disks, i.e. the family $\{F^n|_{D}\}_{n \in \mathbb N}$ is normal. However, we will also show in Theorem \ref{thm:Julia} that the disks $D$ are completely contained in the Julia set of the map $F$. We note that were this note the case, then the existence of a wandering Fatou component would follow immediately from our construction.

In the next section we provide background on polynomial skew-products with a super-attracting fiber. In section (3) we prove a Koenigs-type theorem for
skew products with an attracting invariant fiber. In section (4) we show that, for a smaller class of maps, we can obtain much faster convergence of the Koenigs-map. In
section
(5) we use the convergence estimates to construct, again under additional assumptions, the vertical Fatou disks mentioned in Theorem \ref{thm:main}. Although the maps for
which
we prove Theorem \ref{thm:main} have to satisfy several serious assumptions, we give several explicit examples of such maps. In section (6) we show that, again under
additional
assumptions, the constructed vertical Fatou disks lie entirely in the Julia set, and therefore do not give rise to wandering Fatou components.

{\bf Acknowledgement.} The first author was supported by a SP3-People Marie Curie Actionsgrant in the project Complex Dynamics (FP7-PEOPLE-2009-RG, 248443). This research was partly carried out while the first author visited IMPA in Rio de Janeiro. We thank the institute for its hospitality.

\section{Background}

Let $X$ be a complex manifold, and let $f: X \rightarrow X$ be a holomorphic endomorphism. We say that $z \in X$ lies in the \emph{Fatou set} if there is a neighborhood $U(z)$
on which the family $\{f^n\}_{n \in \mathbb N}$ is normal. A connected component of the Fatou set is called a Fatou component. For rational functions of the Riemann sphere the
Fatou components are well understood. It was proved in by Sulivan \cite{Sullivan1985} that every Fatou component is (pre-) periodic, and the periodic Fatou components were
classified earlier in the works of Fatou, Siegel and Herman.

In two complex dimensions much less is known about Fatou components. There has been considerable progress with respect to the classification of periodic Fatou components, see
for example \cite{BS1991a}, \cite{BS1991b}, \cite{BS1999}, \cite{FS1995}, \cite{LP2014}, \cite{Ueda1994}. However, the existence question for wandering Fatou components
remains
almost untouched. Only for a few specific cases, the non-existence of wandering Fatou components is known, for example for polynomial automorphism that act hyperbolically on
their Julia sets, see \cite{BS1991a}, and for polynomial skew-products near a super-attracting invariant fiber.

A skew-product is a map $F: \mathbb C^2 \rightarrow \mathbb C^2$ of the form
$$
F(t,z) = (f(t), g(t,z)).
$$
The holomorphic dynamics of polynomial skew-products was first studied by Heinemann in \cite{Heinemann} and by Jonsson in \cite{Jonsson}. The topology of Fatou components of skew products has been studied by Roeder in \cite{Roeder}.

Since skew-products map vertical lines
to vertical lines, often one-dimensional tools can be used to construct maps with specific dynamical behavior. For example in \cite{dujardin}, Dujardin used skew-products to
construct a non-laminar Green current, and in \cite{BFP}, Boc-Thaler, Forn{\ae}ss and the first author constructed a Fatou component with a punctured limit set. In both
articles polynomial skew-products were used to construct holomorphic endomorphisms of $\mathbb P^2$ with dynamical phenomena that were previously not known to occur. Hence in
order to study the existence of wandering Fatou components for rational maps in two complex dimensions, it is worthwhile to first investigate whether wandering Fatou components
can occur for polynomial skew-products.

In his PhD thesis Lilov studied the dynamical behavior of polynomial skew-products near an invariant vertical fiber. Since his results were not published elsewhere we will
discuss Lilov's results in a little more detail. Let us assume that $f(0) = 0$, and that $|f^\prime(0)|<1$, so that the orbits of all nearby fibers converge to the invariant
fiber $\{t=0\}$.

Let us further assume that $g$ is a polynomial in $z$, so that
$$
g(t,z) = \alpha_0(t) + \alpha_1(t) z + \cdots + \alpha_d(t) z^d.
$$
In this case the dynamics in the invariant fiber given by the polynomial
$$
p(z) = g(0,z)
$$
is very well understood. Each Fatou component $U$ of $p$ is (pre-)periodic, and every periodic Fatou component is either an attracting basin, a parabolic basin or a Siegel
disk. For each of these cases, and hence for any Fatou component of $p$, the following holds.

\begin{thm}[Bulging of Fatou components]
Given a Fatou component $U$ of $p$, there exists a Fatou component $\Omega$ of $F$ for which $U \subset \Omega$.
\end{thm}

Following Lilov we will refer to the Fatou components of $F$ that contain a Fatou component of $p$ as a \emph{fattened} Fatou component. Recall that the one-dimensional Fatou
component $U$ must at some point be mapped onto either an attracting basin, a parabolic basin, or a Siegel disk. In his thesis Lilov studies the geometry of the fattened Fatou
components in each of these three cases under the assumption that the horizontal dynamics is super-attracting, i.e. when $f^\prime(0) = 0$.

Another important ingredient in the proof of Theorem \ref{thm:lilov} of Lilov is the following.

\begin{lemma}\label{shrinkage}
There exist constants $\delta_1 > 0$, $c>0$ such that if $|t|< \delta_1$ and $D(z,r) \subset \mathbb C_t$ is an arbitrary vertical disk of radius $r > 0$, then $F(t, D(z,r))$
contains a disk $D(g(t,z), r^\prime)$ of radius $r^\prime \ge c r^d$.
\end{lemma}

Applying this lemma repeatedly for the orbit of a vertical disk gives estimates from below for the radii of the images. On the other hand, by studying the geometry of the
fattened Fatou components Lilov also obtains an upper bound on the largest possible vertical disk that can lie in the complement of the fattened Fatou components, depending on
the
distance to the invariant vertical fiber. A combination of these two estimates then leads to the conclusion that the orbit of any vertical disk must intersect one of the
fattened Fatou components. Since the fattened Fatou components are (pre-) periodic, it follows that, in a neighborhood of the invariant fiber, all Fatou components of $F$ are
pre-periodic.

\medskip

The question we study here is whether the argument used by Lilov can also be applied for the geometrically attracting case $0 < |f^\prime(0)| < 1$. We note that the estimate
in
lemma \ref{shrinkage} is very rough. Indeed, this estimate is relevant only when $p$ has a single critical point of order $d-1$, and the $t$-coordinate of the vertical disk
$D(z,r)$ lies in a small neighborhood of this critical point. If this is not the case then much better estimates can be obtained.

If the orbit of a vertical disk does not intersect the fattened Fatou components then this orbit must converge to the Julia set of $p$. If a critical point of $p$ lies in the
Fatou set then the orbit of $B(z,r)$ will eventually stay away from that critical point, and the estimates are much improved. On the other hand, if a critical point lies in
the
Julia set, then it cannot be periodic, and the images $F^n(t, D(z,r))$, which must shrink, will usually be bounded away from that critical point. Hence if the orbit of a
vertical disk does not intersect one of the fattened Fatou components then the radii of the images will shrink much more slowly than Lemma \ref{shrinkage} suggests.

The above discussion leads to the main idea of this paper. In order to obtain a vertical disk whose orbit does not intersect the fattened Fatou component, one must construct a
skew product $F$ and a vertical disk $D$ whose orbit frequently returns very closely to the critical set. We show that this can indeed be achieved for a large class of maps.

\section{A Koenigs-type theorem for attracting skew-products.}\label{section:koenigs}

Let us recall the following classical result.

\begin{thm}[Koenigs, 1870]
Let $f:(\mathbb C, 0) \rightarrow (\mathbb C, 0)$ be the germ of a holomorphic
function with a fixed point at the origin. Suppose $\lambda = f^\prime(0)$
is such that $|\lambda| \neq 0,1$. Then there exists a local change of
coordinates $\varphi$ such that
$$
\varphi(\lambda \cdot z) = f(\varphi(z)).
$$
for all $z$ in a neighborhood of the origin. Moreover, the conformal map
$\varphi$ is unique up to a multiplicative constant.
\end{thm}

We will now prove a result for attracting skew-products that closely resembles
Koenigs Theorem.

\begin{thm}\label{skewkoenigs}
Let $F$ be a skew-product of the form
$$
F(t,z) = (\mu t, g(t,z)),
$$
where $g$ is the germ of a holomorphic function with $g(0,0) = 0$ and $\frac{\partial g(0,0)}{\partial t}
\neq 0$. Let $p(z) = g(0,z)$ be a polynomial with $p'(0) = \la$ and $|\la| >
1$, and let $0< |\mu| < 1$. Then the holomorphic functions
$$
\varphi_n(w) = \pi_2 F^n\left(\frac{w}{\lambda^n}, 0\right)
$$
converge uniformly on compact subsets of $\CC$ to a holomorphic function $\Phi$ which
satisfies the functional equation
\begin{equation}\label{eq:functional}
\Phi(\lambda \cdot w) = p(\Phi(w)).
\end{equation}
\end{thm}

\begin{remark} Theorem \ref{skewkoenigs} follows almost immediately if $F$ can be linearized near the fixed point $(0,0)$. However, in the next sections we will consider
Theorem \ref{skewkoenigs} in the resonant case $\mu = \frac{1}{\lambda}$, where local linearization may not be possible.
\end{remark}

\noindent{\bf Proof of Theorem \ref{skewkoenigs}}
Let us write $g_t(z) = g(t,z)$. By our assumptions we have $g(t,0) = at + O(t^2)$ with $a \neq 0$. Since $g$ is holomorphic there exist $M, A >0$ so that
$$
|g(t,z) - \la z| < M|z|^2+ A|t|
$$
for all $|z|, |t|<1$. Define $R_0=0$ and
$$
R_j = \frac{1 - |\mu\la^{-1}|^{j}}{1 - |\mu\la^{-1}|},
$$
so that $R_{j+1} = R_j + |\mu\la^{-1}|^{j}$ for all $j \ge 0$. Since $|\mu|<|\la|$ the sequence
$(R_j)$ is bounded by
$$
R_\infty = \frac{1}{1 - |\mu\la^{-1}|}.
$$

\begin{lemma} For all $j \leq n$:
$$
|\varphi_{n,j}(w)| \le A|w|\left(R_j|\lambda|^{-n+j-1} + AM|\lambda|^{-2n+2j-5}\right).
$$
\end{lemma}

\begin{proof}
We use induction on $j$. The inequality in the lemma is trivially satisfied for $j = 0$. Suppose that the holds for certain $j < n$. Then we have that
$$
\begin{aligned}
|\varphi_{n,j+1}(w)| &= |g_{\mu^{j} w/\la^n}(\varphi_{n,j}(w))|\\
&< |\la||\varphi_{n,j}(w)| + M|\varphi_{n,j}(w)|^2 + A|\mu^j w/\la^n|\\
&< |\la| |Aw|\left(R_j|\lambda|^{-n+j-1} + AM|\lambda|^{-2n+2j-5}\right)\\
&+  M|Aw|^2\left(R_j|\lambda|^{-n+j-1} + AM|\lambda|^{-2n+2j-5}\right)^2 + A|\mu^j w/\la^n|\\
&< |Aw||\la|^{-n+j}(R_j+A\mu^j/\la^j)+ A^2M|w|^2|\lambda|^{-2n+2j-2} \left(R_j+ AM|\lambda|^{-n+j-4}\right)^2 \\
&+ A^2M|w||\la|^{-2n+2j-4} \\
&< |Aw|\left(R_{j+1}|\lambda|^{-n+j} + AM|\lambda|^{-2n+2j-3}\right).
\end{aligned}
$$
Here the last inequality is obtained by choosing $|w|$ small enough so that:
\begin{eqnarray*}
|w|< \frac{|\la|^{-1} -|\la|^{-2}}{\left(R_\infty+ AM\right)^2 } .
\end{eqnarray*}
\end{proof}

Now we can prove:

\begin{lemma}
There exists $K_1>0$ and $\delta_0>0$ so that for $|w|< \delta_0$ one has
$$
|\varphi_{n,j}(w)| \le K_1 \frac{|w|}{|\lambda^{n-j}|},
$$
for all $n \in \mathbb N$ and all $j = 1, \ldots, n$.
\end{lemma}
\begin{proof}
We can choose $K_1$ so that
$$
A\left(R_j|\lambda|^{-1} + AM|\lambda|^{-n+j-5}\right) < A(R_{\infty}/|\la| + AM) < K_1,
$$
and the lemma follows.
\end{proof}

Our goal is to estimate $|\varphi_{n,j}(w) - \varphi_{n+1,j+1}(w)|$, and then
to apply these estimates to the case $j = n$. Let us define holomorphic
functions $F_{n,j} = F_{n,j}(w)$ so that
\begin{eqnarray*}
\varphi_{n,j}(w) &=& \frac{aw}{\la^n}\left( \mu^{j-1} + \mu^{j-2}\la + \ldots + \la^{j-1} \right) + F_{n,j}\\
 &=& aw \frac{\la^{j-n} - \la^{-n}\mu^{j}}{\la - \mu} + F_{n,j}.
\end{eqnarray*}
Define $L_{n,j}(w) = aw \frac{\la^{j-n} - \la^{-n}\mu^{j}}{\la - \mu}$. Note that
\begin{equation}
\varphi_{n+1,j+1}(w) - \varphi_{n,j}(w) = aw\frac{\mu^j}{\lambda^{n+1}} +
F_{n+1,j+1} - F_{n,j},
\end{equation}
and that
$$
|F_{n,j+1}| \le \lambda |F_{n,j}| + K_2(|\varphi_{n,j}(w)|^2 +
\left|\frac{\mu^jw}{\lambda^{n}}\right|^2 + \left|\frac{\mu^jw}{\lambda^{n}}\varphi_{n,j}(w)\right|),
$$
where $K_2>0$ is such that
$$
|g(t,z) - \lambda z - at| < K_2(|z|^2+|zt|+|t|^2)
$$
for all $|z|,|t| < 1$.

\begin{lemma}
There exists a $K_3>0$ and $\delta_1 > 0$, so that for all $|w| < \delta_1$ we have that
$$
|F_{n,j}| \le K_3{|\la|^{-2n+2j}}.
$$
for all $n$ and all $j \le n$.
\end{lemma}

\begin{proof}
We use induction on $j$. Clearly the inequality is satisfied for $j = 1$.
Suppose that the lemma is satisfied for certain $j < n$. Then we have
$$
\begin{aligned}
|F_{n,j+1}| &\le \la |F_{n,j}| + K_2\left(|\varphi_{n,j}(w)|^2 + \left|\frac{\mu^jw}{\la^{n}}\varphi_{n,j}(w)\right|+ \left|\frac{\mu^jw}{\la^{n}}\right|^2\right)\\
&\le K_3{|\la|^{-2n+2j}} + K_2 \left( K_1^2 \frac{|w|^2}{|\lambda^{2n-2j}|}+K_1 \frac{|w|}{|\lambda^{n-j}|}\frac{|\mu^jw|}{|\la|^{n}} +
\frac{|\mu^{2j}w^2|}{|\la|^{2n}}\right)\\
&\le K_3{|\la|^{-2n+2j}} + K_2|\la|^{-2n+2j}|w|^2( K_1^2+K_1+1)\\
&\le K_3{|\la|^{-2n+2j+2}},
\end{aligned}
$$
where the second and last inequalities hold uniformly over all $j$ and $n$ when
$K_3$ is chosen sufficiently large and $\delta_1$ sufficiently small.
\end{proof}

Let $\epsilon >0$ so that $|\la| + \epsilon < |\la|^2$.

\begin{lemma}
There exists a $\delta_2 > 0$ so that for $|w|< \delta_2$ we have
$$
|F_{n,j+1} - F_{n+1,j+2}| \le \frac{|\mu\la|^j}{|\la^{2n}|} +
\left(|\lambda| + \frac{\epsilon}{|\lambda^{n-j}|}\right) \cdot |F_{n,j} -
F_{n+1,j+1}|.
$$
for all $j \le n \in \mathbb N$.
\end{lemma}
\begin{proof}
To make our estimates more readable we will write $\varphi_{n,j}$ for $\varphi_{n,j}(w)$. We have that
\begin{eqnarray*}
\vp_{n,j+1} &=& g_{\mu^jw/\la^n} \circ \vp_{n,j} = g_0(\vp_{n,j})+ a\frac{\mu^jw}{\la^n} + h\left(\frac{\mu^jw}{\la^n},\vp_{n,j}\right)
\end{eqnarray*}
where
$g_t(z) = g_0(z) + at + h(t,z)$, if we have $|t|,|z|<1$, then $|h(t,z)|< M(|t|^2+|tz|)$.
Hence we obtain:
$$
\begin{aligned}
F_{n,j+1} &= \la\vp_{n,j}+ (g_0-\la Id)(\vp_{n,j})+ a\frac{\mu^jw}{\la^n} - L_{n,j+1} + h\left(\frac{\mu^jw}{\la^n},\vp_{n,j}\right) \\
 &= \la F_{n,j} + \la L_{n,j}+ (g_0-\la Id)(\vp_{n,j})+ a\frac{\mu^jw}{\la^n} - L_{n,j+1} + h\left(\frac{\mu^jw}{\la^n},\vp_{n,j}\right)
\end{aligned}
$$
It follows that
$$
\begin{aligned}
F_{n,j+1} &- F_{n+1,j+2} = \la (F_{n,j}-F_{n+1,j+1}) + \la(L_{n,j}-L_{n+1,j+1})+ (g_0-\la Id)(\vp_{n,j})\\
&-(g_0-\la Id)(\vp_{n+1,j+1})+ a\frac{\mu^jw}{\la^n} - a\frac{\mu^{j+1}w}{\la^{n+1}} - L_{n,j+1} + L_{n+1,j+2}+ \\
&+h\left(\frac{\mu^jw}{\la^n},\vp_{n,j}\right) - h\left(\frac{\mu^{j+1}w}{\la^{n+1}},\vp_{n+1,j+1}\right)
\end{aligned}
$$
Recalling the definition of $L_{n,j}$ we obtain:
$$
\begin{aligned}
|F_{n,j+1} & - F_{n+1,j+2}| \le |\la| |F_{n,j}-F_{n+1,j+1}| + M|\vp_{n,j}-\vp_{n+1,j+1}||\vp_{n,j}+\vp_{n+1,j+1}|+ \\
&+|h\left(\frac{\mu^jw}{\la^n},\vp_{n,j}\right)| + |h\left(\frac{\mu^{j+1}w}{\la^{n+1}},\vp_{n+1,j+1}\right)|\\
&\le |\la| |F_{n,j}-F_{n+1,j+1}| + 2MK_1\frac{|w|}{|\la|^{n-j}}|L_{n,j}+ F_{n,j}-L_{n+1,j+1}-F_{n+1,j+1}|+ \\
&+M|\frac{\mu^jw}{\la^n}|^2+M|\frac{\mu^jw}{\la^n}||\vp_{n,j}| + M|\frac{\mu^{j+1}w}{\la^{n+1}}|^2+M|\frac{\mu^{j+1}w}{\la^{n+1}}||\vp_{n+1,j+1}|\\
&\le \left(|\la|+2MK_1\frac{|w|}{|\la|^{n-j}}\right) |F_{n,j}-F_{n+1,j+1}| +  2MK_1\frac{|w|}{|\la|^{n-j}}|aw\mu^j\la^{-n-1}|\\
&+M|\frac{\mu^{2j}w^2}{\la^{2n}}|(1+\mu^2\la^{-2})+MK_1|w|^2|\mu^j||\la^{j-2n}|(1 + \mu\la^{-1})\\
&\le \left(|\la|+2MK_1\frac{|w|}{|\la|^{n-j}}\right) |F_{n,j}-F_{n+1,j+1}| +  2MK_1|a|\frac{|w|^2|\mu|^j}{|\la|^{2n-j+1}}\\
&+2M|\frac{\mu^{2j}w^2}{\la^{2n}}|+2MK_1|w|^2|\mu^j||\la^{j-2n}|
\end{aligned}
$$
By choosing $\delta_2$ small enough we make sure that
$$
2MK_1|a||w|^2|\la|^{-1} < 1/3, \; \; \; 2M|w|^2<1/3, \; \; \; \mathrm{and} \; \; \; 2MK_1|w|^2<1/3.
$$
Hence we obtain:
$$
2MK_1|a|\frac{|w|^2|\mu\la|^j}{|\la|}+2M|\mu^{2j}w^2|+2MK_1|w|^2|\mu^j||\la^{j}|< 2/3|\mu\la|^j + 1/3|\mu^{2j}| <|\mu\la|^j.
$$
By choosing $\delta_2>0$ smaller if necessary we can also guarantee that $2MK_1|w| < \epsilon$, which gives the desired inequality.
\end{proof}

\begin{lemma}
For large enough $M_2 > 0$ and $|w| < \delta_2$ we have that
$$
|F_{n,j} - F_{n+1,j+1}| \le M_2\frac{(|\la|+\epsilon)^j}{\la^{2n}} + \frac{1}{|\la|^{2n}}\left(\frac{(|\la|+\epsilon)^j-|\mu|^j|\la|^j}{|\la|+\epsilon-|\mu\la|}\right)
$$
for all $k \le n \in \mathbb N$.
\end{lemma}
\begin{proof}
For $j=0$ the statement is immediate. The result is obtained by induction on $j$, using the previous lemma.
\end{proof}

\begin{cor}
$$
|F_{n,n} - F_{n+1,n+1}| \le M_2\frac{(|\la|+\epsilon)^n}{|\la|^{2n}} + \frac{1}{|\la|^{2n}}\left(\frac{(|\la|+\epsilon)^n-|\mu|^n|\la|^n}{|\la|+\epsilon-|\mu\la|}\right)
$$
\end{cor}

\begin{cor}
There exists an $M_3>0$ so that for $|w|< \delta_2$ we have that
$$
|\varphi_n(w) - \varphi_{n+1}(w)| \le \frac{M_3}{\theta^{n}}
$$
for some $\theta >1$.
\end{cor}

It follows that the maps $\varphi_n$ converge uniformly on some neighborhood
$U= U(0)$, which we may assume to be a disk centered at the origin. Now let $w
\in \mathbb C$ and let $k \in \mathbb N$ be such that $\frac{w}{\lambda^k} \in
U$. Then it follows that
$$
\begin{aligned}
\varphi_n(w) &= \pi_2 F^n(\frac{w/\lambda^k}{\lambda^{n-k}},0) = \pi_2 F^k(F^{n-k}(\frac{w/\lambda^k}{\lambda^{n-k}},0))\\
 &= \pi_2 F^k(\frac{\mu^{n-k}w}{\lambda^n}, \varphi_{n-k}(\frac{w}{\lambda^k})).
\end{aligned}
$$
Note that $\frac{\mu^{n-k}w}{\lambda^n} \rightarrow 0$ and $\varphi_{n-k}(\frac{w}{\lambda^k}) \rightarrow \Phi(\frac{w}{\lambda^k})$ as $n \rightarrow \infty$. Since $\pi_2
\circ F^k$ is a continuous map independent of $n$, it follows that $\varphi_n(w)$ converges to $\pi_2 F^k(0,\Phi(\frac{w}{\lambda^k})) = g_0^k(\frac{w}{\lambda^k})$. Hence it
follows that the maps $\varphi_n$ converge uniformly on any compact
set, and that the limit map $\Phi$ is defined globally and satisfies
$$
\Phi(\lambda \cdot w) = g_0(\Phi(w)).
$$
This completes the proof.
\hfill $\square$

\section{Faster convergence for degenerate resonant skew-products}\label{section:faster}

Again we consider polynomial skew-products of the form
$$
F(t,z) = (\mu t, g(t,z)),
$$
but now we assume that $g$ is a polynomial, and that $\mu =\frac{1}{\lambda}$, where $\lambda = \frac{\partial g}{\partial z}(0,0) \neq 0$ as before. Let us write
$$
F(t,z) = \left(\frac{t}{\lambda}, p(z) + a_1(z) t + \cdots + a_m(z) t^m\right).
$$
Let $x_0 \in \mathbb C$ satisfy $p^k(x_0) = 0$ for certain $k \in \mathbb N$, and define
$$
\varphi_{n,j}(w) = F^j(\frac{w}{\lambda^n}, x_0),
$$
and
$$
\varphi_n = \varphi_{n,n}.
$$
By our results from the previous section it follows that the functions $\varphi_n$ converge to a holomophic function $\Phi$, uniformly on compact subsets of $\mathbb C$.

\begin{lemma}
The function $\Phi$ found in Theorem \ref{skewkoenigs} depends only on the polynomial $p$, the point $x_0$ and the constant $a_1(0)$.
\end{lemma}
\begin{proof}
Note that $\varphi_n(0) = 0$ and that $\frac{d}{d w} \varphi_n (0)$ is independent of all other coefficients of $g$ for every $n$. Hence it follows that $\Phi^\prime(0)$ is
also independent of the other coefficients of $g$. Since the Koenigs map $\Phi$ is unique up to a multiplicative constant, the statement follows.
\end{proof}

While the other coefficients have no effect on the limit function $\Phi$, they will be important for us. Choosing them carefully will allow us to obtain a stronger rate of
convergence.

\medskip

For $j \ge k$, $\varphi_{n,j}$ is a holomorphic function with $\varphi_{n,j}(0) = 0$. Therefore, for $j \ge k$ we can write
$$
\varphi_{n,j}(w) = C_{n,j} + D_{n,j} + E_{n,j},
$$
where the function $C_{n,j} = C_{n,j}(w)$ contains all terms that are linear in $w$, $D_{n,j} = D_{n,j}(w)$ contains all terms that are quadratic in $w$, and $E_{n,j} =
E_{n,j}(w)$ contains all higher order terms. Writing
\begin{equation}\label{three}
F(t,z) = (\frac{t}{\lambda}, \lambda z + a t + \gamma z^2 + \tau zt + b t^2 + O(z^3, z^2t, zt^2, t^3)),
\end{equation}
we obtain for $j\ge k$:
$$
\varphi_{n,j+1}(w) = \lambda C_{n,j} + \lambda D_{n,j} + \frac{aw}{\lambda^{n+j}} + \gamma C_{n,j}^2 +  \frac{\tau C_{n,j}w}{\lambda^{n+j}} +  \frac{bw^2}{\lambda^{2n+2j}} +
E_{n,j+1}.
$$
Hence we have
$$
\begin{aligned}
C_{n,j+1} & = \lambda C_{n,j} + \frac{aw}{\lambda^{n+j}},\\
D_{n,j+1} & = \lambda D_{n,j} + \gamma C_{n,j}^2 +  \frac{\tau C_{n,j}w}{\lambda^{n+j}} +  \frac{bw^2}{\lambda^{2n+2j}}.
\end{aligned}
$$
It follows that we can find a closed form for $C_{n,j}$, for $j \ge k$, of the form
$$
\begin{aligned}
C_{n,j} = w\left(Y_1 \lambda^{-n+j} + Y_{-1} \lambda^{-n-j} \right),
\end{aligned}
$$
Where $Y_1$ and $Y_{-1}$ are independent of $n$ and $j$. It follows that $C_{n,j}^2$ is of the form
$$
C_{n,j}^2 = w^2\left( \tilde{Y}_2 \lambda^{-2n+2j} + \tilde{Y}_0\lambda^{-2n} + \tilde{Y}_{-2} \lambda ^{-2n-j}\right),
$$
Therefore there exist constants $X_2, X_1, X_0, X_{-1}, X_{-2} \in \mathbb C$ so that
$$
D_{n,j} = w^2\left( X_2 \lambda^{-2n+2j} + X_1 \lambda^{-2n+j} + X_0 \lambda^{-2n} + X_{-1} \lambda^{-2n-j} + X_{-2} \lambda^{-2n-2j} \right),
$$
again for $j \ge k$.

\medskip

Note that the constant $X_1$ depends affinely on $\tau$, $b$ and $\gamma$. Hence for generic choices of $p$ and $x_0$ we can find $b, \tau \in \mathbb C$ so that $X_1 = 0$.

\begin{defn}
If $F$ is of the form \eqref{three} and $X_1=0$, then we say that $F$ is a \emph{degenerate resonant skew-product}.
\end{defn}

\begin{thm}\label{squared}
Let $F$ be a degenerate resonant skew-product, and let $\epsilon>0$. Then we can find $\delta>0$ such that
$$
|\varphi_{n+1}(w) - \varphi_{n}(w)| \le |\lambda^{(\epsilon-2)n}|.
$$
whenever $|w| < \delta$.
\end{thm}
\begin{proof}
By our assumptions we have that
$$
\begin{aligned}
D_{n,j} =  w^2 \cdot \left(X_2 \lambda^{-2n+2j} + X_0 \lambda^{-2n} + X_{-1} \lambda^{-2n-j} + X_{-2} \lambda^{-2n-2j} \right).
\end{aligned}
$$
Now notice first that
$$
C_{n+1, j+1} - C_{n,j} = \hat{Y}_{-1} \lambda^{-n-j} + \hat{Y}_{-2} \lambda^{-n-j-2},
$$
and due to our assumptions we have
$$
D_{n+1,j+1} - D_{n,j} = w^2 \left( \hat{X_0} \lambda^{-2n} + \hat{X}_{-1} \lambda^{-2n-j} + \hat{X}_{-2} \lambda^{-2n-2j}\right).
$$
Hence from the above equations we obtain for $|w|$ sufficiently small that
$$
|C_{n+1,n+1} + D_{n+1, n+1} - C_{n,n} + D_{n,n}| \le |\lambda^{-2n}|.
$$
What remains is to find similar estimates for $E_{n,j}$.

\begin{lemma}
For $|w|$ small enough we have that
$$
|E_{n,j}| \le |w|^2 |\lambda|^{-3n+3j},
$$
for all $k \le j \le n$.
\end{lemma}
\begin{proof}
We prove the statement by induction on $j$. Note that the required inequality holds for $j = k$ as long as $w$ is sufficiently small. Now suppose that
$$
|E_{n,j}| \le |w|^2 |\lambda|^{-3n+3j},
$$
for certain $n> j \ge k$. Then we have that
\begin{equation}\label{splitting}
E_{n,j+1} = \lambda E_{n,j} + \gamma \left(2C_{n,j} + D_{n,j} + E_{n,j}\right)\cdot\left(D_{n,j} + E_{n,j}\right) + \cdots.
\end{equation}
Hence there exists a constant $K$ so that for $w$ sufficiently small we have
$$
|E_{n,j+1}| \le |w|^2|\lambda|^{-2} |\lambda|^{-3n+3j+3} + K |w|^3 |\lambda|^{-3n+3j+3}.
$$
By making $w$ smaller if necessary we obtain
$$
|E_{n,j+1}| \le |w|^2 |\lambda|^{-3n+3j+3}.
$$
\end{proof}

\begin{lemma}
Let $\epsilon>0$. Then for $|w|< \delta_6$ we have that
$$
|E_{n+1,j+1} - E_{n,j}| \le |w|^2 |\lambda|^{-3n+(1+\epsilon)j},
$$
for all $k \le j \le n$.
\end{lemma}
\begin{proof}
We again prove the statement by induction. The first step
$$
|E_{n+1,k+1} - E_{n,k}| \le |w|^2 |\lambda|^{-3n+(1+\epsilon)k}
$$
follows directly from our previous lemma.

Suppose that $|E_{n+1,j+1} - E_{n,j}| \le  |\lambda|^{-3n+(1+\epsilon)j}$ for certain $k \le j \le n$. Then it follows from Equation \ref{splitting}, our induction hypothesis
and our previous estimates that
$$
\begin{aligned}
|E_{n+1,j+2} - E_{n,j+1}| & \le |\lambda| |E_{n+1,j+1} - E_{n,j}| + K_2 |w|^3 |\lambda|^{-3n+j}\\
&\le |\lambda|^{-\epsilon}\cdot |w|^2 |\lambda|^{-3n+(1+\epsilon)j} + K_2 |w|^3 |\lambda|^{-3n+j}\\
&\le |w|^2 |\lambda|^{3n+(1+\epsilon)j},
\end{aligned}
$$
where the last inequality holds (uniform over all $j$ and $n$) for $w$ sufficiently small.
\end{proof}

\noindent {\bf Completing the proof of Theorem \ref{squared}:}
Plugging $j = n$ into the estimate of the last lemma gives
$$
|E_{n+1,j+1} - E_{n,j}| \le |w|^2 |\lambda|^{(\epsilon-2)n},
$$
which combined with our earlier estimates on $|C_{n+1,j+1} + D_{n+1,j+1} - C_{n,j} - D_{n,j}|$ gives the required estimate on $|\varphi_{n+1,j+1}(w) - \varphi_{n,j}(w)|$, and
completes the proof of Theorem \ref{squared}.
\end{proof}

\section{Vertical Fatou disks}

Again we consider skew-products of the form
$$
F(t,z) = (\frac{t}{\lambda}, g(t,z)),
$$
with $|\lambda| > 1$. Now we make the additional assumption that
\begin{equation}\label{condition2}
g(t,z) = p(z) + q(t),
\end{equation}
where $p(t) = \lambda t + h.o.t.$ and $q(t) = at + bt^2 + h.o.t.$ with $a \neq 0$. We assume that $F$ is a degenerate resonant skew-product, so that Theorem \ref{squared} holds. Finally we
assume that $p$ has a critical point $x_0$ of order at least $3$, for which $p^k(x_0) = 0$.

Let $w_0 \in \mathbb C$ be such that $\Phi(w_0) = x_0$, i.e. $\pi_2 \circ F^n(\frac{w_0}{\lambda^n}, x_0) \rightarrow x_0$ as $n \rightarrow \infty$. Such a $w_0$ might not be
unique but it can be found since $x_0$ necessarily lies in the Julia set of $p$, and hence not in the exceptional set of $p$. Therefore $x_0$ is contained in the union
$$
\bigcup_{n \in \mathbb N} p^n(U),
$$
for any neighborhood $U$ of the repelling fixed point $0 \in \mathcal{J}(p)$. The fact that $w_0$ exists now follows from the functional equation
$$
\Phi(\lambda \cdot w) = p(\Phi(w)).
$$
We will refer to the complex lines $\{t = \frac{w_0}{\lambda^n}\}$ as \emph{critical fibers}.

\begin{defn}\label{def:disks}
We refine the vertical disks $D_n$ as follows:
$$
D_n := \{ (\frac{w_0}{\lambda^n}, z) \mid |z - x_0| < |\lambda^{-3n/4}|\}.
$$
\end{defn}

We will prove that for $n$ sufficiently large, the forward orbits of the disks $D_n$ all avoid the bulged Fatou components of $F$.

\begin{lemma}\label{lemma:nested}
For $n$ sufficiently large we have that
$$
F^n(D_n) \subset D_{2n}.
$$
\end{lemma}
\begin{proof}
Since $x_0$ is a critical point of order at least $3$ and we assumed that $g(t,z) = p(z) + q(t)$, it follows that $F(D_n)$ is contained in a vertical disk centered at the
point
$F(\frac{w_0}{\lambda^n},x_0)$ of radius $M_1 |\lambda^{-3n}|$, for some constant $M_1>0$ which is uniform over $n \in \mathbb N$.

For any $\epsilon >0$ there exists an $m$ independent of $n$ so that for $j = k, \ldots , n-m$ the $z$-coordinates of the images $F^j(D_n)$ lie in the $\epsilon$-neighborhood
of the repelling fixed point $0 = p(0)$ with multiplier $\lambda$.

Hence by choosing $\epsilon>0$ small enough we see that for $n$ large enough the image $F^{n}(D_n)$ is contained in the disk centered at $F^n(\frac{w_0}{\lambda^n}, x_0)$ of
radius
$$
M_2 |\lambda^{-3n}|\cdot |\lambda|^{\frac{5n}{4}} = M_2 |\lambda|^{-\frac{7n}{4}}.
$$
The first coordinate of $F^n(\frac{w_0}{\lambda^n}, x_0)$ is $\frac{w_0}{\lambda^{2n}}$, and the distance to $x_0$ of the second coordinate is at most $M_3
|\lambda|^{(\epsilon
- 2)n}$ by Theorem \ref{squared}. Hence by the definition of $D_{2n}$ we have that for $n$ large enough the image $F^n(D_n)$ is contained in the disk $D_{2n}$.
\end{proof}

\begin{remark}
An immediate consequence of Lemma \ref{lemma:nested} is that for sufficiently large $n \in \mathbb N$ we have that $F^{(2^\ell-1)n}(D_n) \rightarrow (0,x_0)$ as $\ell
\rightarrow \infty$.
\end{remark}

Our main result is now a quick consequence.

\begin{thm}\label{disks}
There exists an $N \in \mathbb N$ so that for every $n \ge N$ the forward images of the disks $D_n$ never intersect the bulged Fatou-components of $F$.
\end{thm}
\begin{proof}
Notice that $F^{j_l}(D_n) \rightarrow (0,x_0)$ implies that the images $F^{j}(D_n)$ cannot intersect one of the bulged Fatou components of $F$. This follows from the fact that
the Fatou components of $p$ are (pre-)periodic, and the classification of periodic Fatou components. Each periodic Fatou component of $p$ is either an attracting basin, a
parabolic basin or a Siegel disk, hence there can never be a subsequence for which an orbit in a Fatou component converges to a non-periodic point in the Julia set. It follows
immediately from the attraction in the horizontal direction that the same holds for the corresponding bulged Fatou components of $F$.
\end{proof}

\begin{lemma}
For each sufficiently large $n \in \mathbb N$ the forward images $F^j(D_n)$ accumulate at a closed and invariant subset of the form
$$
\Omega = \{x_{-\ell}\}_{\ell=-k}^{+\infty},
$$
where
\begin{equation}\label{conditions}
p(x_{-\ell}) = x_{-\ell+1} \; \; \;  \mathrm{and} \; \; \; \lim_{\ell \rightarrow \infty} x_{-\ell} = 0.
\end{equation}
Moreover, for each choice of a set $\Omega$ satisfying the conditions in \eqref{conditions} we can find a vertical disk with its $\omega$-limit set equal to $\Omega$.
\end{lemma}
\begin{proof}
Recall that for each sufficiently large $n$ the image $F^n(D_n)$ is contained in the disk $D_{2n}$.  By the functional equation \eqref{eq:functional} it follows that
$$
\pi_2 F^{2^j n -\ell}(D_{2^j n})
$$
converges to a point $x_{-\ell}$ as $j \rightarrow \infty$. Hence the $\omega$-limit set contains the required sequence $\{x_{-\ell}\}$ and the point $0$ which lies in the
closure. To see that the $\omega$-limit set contains no other points, notice that the images
$$
\pi_2 F^{2^j n -\ell}(D_{2^j n})
$$
converge to $0$ as $\ell \rightarrow \infty$, uniformly over all $j$ as long as $2^{j-1}n > \ell$.

\medskip

For the converse, let $\{x_n\}$ be an inverse orbit of $p$ which converges to the repelling fixed point $0$. Since $p^\prime(0) \neq 0$, the function $p$ is injective in a
small disk $U = U(0)$. Let $N \in \mathbb N$ be such that $x_n \in U$ for $n \ge N$, which implies that $x_{-N}$ determines the entire inverse orbit.

The holomorphic map $\Phi$ satisfies $\Phi(0) = 0$. Hence we can find $w \in \mathbb C$ so that $\Phi(\frac{w}{\lambda^N}) = x_{-N}$ and so that $\Phi(\frac{w}{\lambda^{N+j}})
\in U$ for all $j \in \mathbb N$. The claim follows.
\end{proof}

\begin{remark}\label{example}
It is not hard to generate many explicit examples of degenerate resonant skew-products for which $p$ has a critical point $x_0$ of order at least $3$ with $p^k(x_0) = 0$. Consider for example the
polynomials $p(z) = 2(z+1)^d - 2$, for $d \ge 4$ even. Here the multiple critical point $z=-1$ of order $d-1$ is mapped in two steps to the repelling fixed point $z=0$. Since
the derivative $p^\prime(0) = 2d$, we can consider skew products of the form
$$
F(t,z) = (\frac{t}{2d}, 2(z+1)^d -2 + at + bt^2).
$$
Given $a \neq 0$ the value for $b$ can be explicitly computed. For example if $d=4$ and $a = 1$ then computation shows that we should take $b = -\frac{641}{4165}$ in order for
$F$ to be a degenerate resonant skew-product. Two computer pictures of vertical fibers for this example are shown in Figure 1. On the left a critical fiber $t = w_0$ with a clearly visible Fatou disk. On the right a nearby fiber $\{t=t_0\}$, for $t_0 \sim w_0 + 0.001$, where the one-dimensional filled Julia set
$$
\{ z \in \mathbb C \mid \{F^n(t_0, z)\} \; \; \mathrm{bounded} \; \}
$$
seems to have empty interior. This suggests that the Fatou disk lies in the Julia set of the $2$-dimensional map, which we prove in the next section.
\end{remark}

\begin{figure*}\label{pictures}
\centerline{%
\includegraphics[height=6cm]{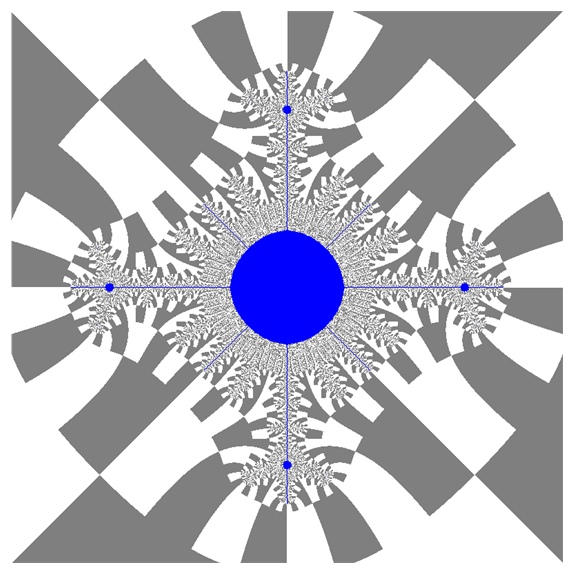}
\includegraphics[height=6cm]{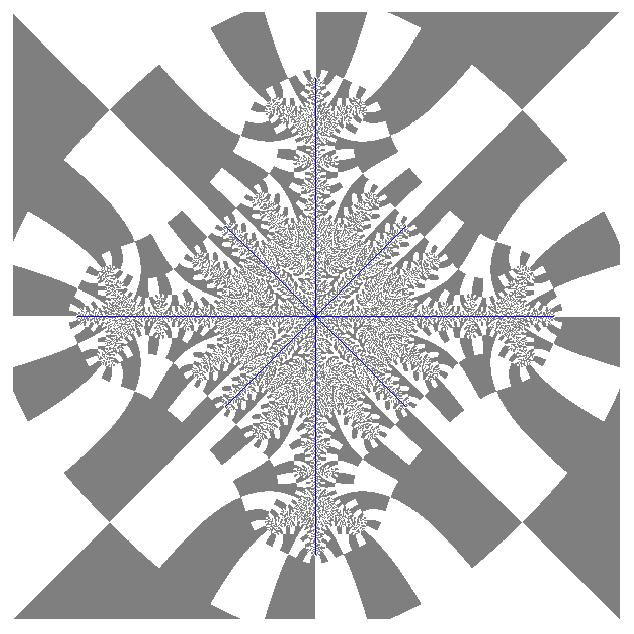}
}%
\caption{A vertical Fatou disk and a nearby fiber}
\end{figure*}

The notion of a Fatou-map was introduced by Ueda in \cite{Ueda2008}.

\begin{defn}
Let $f: X \rightarrow X$ be a holomorphic endomorphism of a complex manifold $X$. A holomorphic disks $D \subset  X$ is a \emph{Fatou disk} for $f$ if the restriction of
$\{f^n\}$ to the disk $D$ is a normal family.
\end{defn}

\begin{lemma}
The disks $D_n$ are Fatou disks.
\end{lemma}
\begin{proof}
The $t$-coordinates converge to zero and the $z$-coordinates stay bounded. Hence normality follows by Montel's Theorem.
\end{proof}

\section{The disks lie in the Julia set}

As in the previous section we consider a degenerate resonant skew-product of the form
$$
F(t,z) = (\frac{t}{\lambda}, p(z) + q(t)).
$$
As before, we assume that $p$ has a the critical point $x_0$ of order
$s
\ge 3$ which by $p^k$ is mapped to the repelling fixed point $0 = p(0)$ with multiplier $\lambda$. In this section we make the additional assumption that $p$ has no
critical points besides $x_0$. Note that the examples mentioned in Remark \ref{example} satisfy this assumption.

We let $D_n$ be the vertical disk centered at $(\frac{w_0}{\lambda^n},x_0)$ as defined in \ref{def:disks}. Recall from Lemma \ref{lemma:nested} that there exists an $N \in \mathbb N$ for which $F^n(D_n) \subset D_{2n}$ for all $n \ge N$.

\begin{thm}\label{thm:Julia}
For $n \ge N$ the disks $D_n$ lie entirely in the Julia set of $F$.
\end{thm}
\begin{proof}
We will prove that the disk $D_N$ is entirely contained in the Julia set of $F$, which implies the same statement for all $n \ge N$.

Suppose for the purpose of a contradiction that the disk $D_N$ intersects the Fatou set of $F$. Then any such intersection point is contained in an open neighborhood $U$ which
is relatively compact in the Fatou set, and on which $(F^j)_{j \in \mathbb N}$ is a normal family. Recall that the $\omega$-limit set of the disk $D_n$ is contained in the
Julia set of $p$, which has no interior. Therefore the diameter of $F^j(U)$ must converge to $0$ as $j \rightarrow \infty$.

Let $v \in \mathbb C$ with $v \neq w_0$ for which $U$ intersects the line $\{t=\frac{v}{\lambda^N}\}$, say in the point $(t_N,z_N) = (\frac{v}{\lambda^N},z_N)$. Since $\Phi$
is
non-constant we may assume that $\Phi(v) \neq \Phi(w_0)$. Let us write $y = \Phi(v)$, and $(t_{N+j},z_{N+j}) = F^j(t_N,z_N)$. Let $\epsilon >0$ small enough so that $|y - x_0|
> 3\epsilon$.

Since the diameter of $F^j(U)$ converges to zero, there exists a $J_1 \in \mathbb N$ so that $\mathrm{diam}(F^j(U)) < \epsilon$ for $j \ge J_1$. Since $F^n(D_n) \subset
D_{2n}$
for all $n \ge N$ and the disks $D_n$ converge to the point $(0,x_0)$, it follows that there exists an $L \in \mathbb N$ so that for $\ell \ge L$ one has that $F^{(2^\ell -
1)N}(D_N)$ is contained in the ball centered at $(0,x_0)$ with radius $\epsilon$. It follows that for $\ell \ge L$ and $j = (2^\ell -1)N \ge J_1$ we have that
$$
d(F^j(U), (0,y)) > \epsilon.
$$
Recall that $\pi_2 \circ F^j(\frac{v}{\lambda^j},x_0)$ converges to $y$, hence there exists $J_2 \ge J_1$ so that for $j \ge J_2$ we have that
$$
d(F^j(\frac{v}{\lambda^j},x_0), (0,y)) < \frac{\epsilon}{2}.
$$
It follows that
\begin{equation}\label{eq:number9}
\|(t_{2^{\ell+1}N}, z_{2^{\ell+1}N}) -  F^{2^\ell n}(t_{(2^\ell)N}, x_0) \| > \frac{\epsilon}{2}.
\end{equation}
We also have
$$
(t_{2^{\ell+1}N}, z_{2^{\ell+1}N}) = F^{2^\ell N}(t_{2^\ell N}, z_{ 2^\ell N}),
$$
and $x_0$ is a critical point of $p$ of order $s-1$, which implies that $|p(z) - p(x_0)| \le M|z-x_0|^s$ for $|z - x_0|$ small and certain $M>0$. Choosing
$|\tilde{\lambda}| > |\lambda|$ we therefore obtain
\begin{equation}\label{eq:number10}
\|z_{2^{\ell+1}N} - \pi_2 F^{2^\ell n}(t_{(2^\ell)N}, x_0)\| \le |\tilde{\lambda}|^{2^\ell N} \cdot M \cdot |z_{2^\ell N} - x_0|^s,
\end{equation}
for $n$ large enough. Combining Equations \eqref{eq:number9} and \eqref{eq:number10} we obtain
\begin{equation}\label{eq:distance}
|z_{2^\ell N} - x_0| \ge \delta_0 |\tilde{\lambda}|^{-\frac{2^\ell N}{s}},
\end{equation}
for certain $\delta_0 > 0$.

Equation \eqref{eq:distance} induces an estimate from below on $|p^\prime(z_{2^\ell N})|$. Since $x_0$ is a critical point of order $s-1$, there exists an $\delta_1>0$ so that
$$
|p^\prime(z)| \ge \delta_1 \cdot |z-x_0|^{s-1}
$$
whenever $|z - x_0|$ is sufficiently small. It follows that
$$
|p^\prime(z_{2^\ell N})| \ge \delta_2 \cdot |\tilde{\lambda}|^{\frac{1-s}{s}2^\ell N}.
$$
Note that
$$
DF^j(t_n, z_n) = \left[
\begin{matrix}
f^{\prime}(t_{n+j-1}) & 0\\
q^\prime (t_{n+j-1}) & p^\prime (z_{n+j-1})
\end{matrix}\right]
\cdots
\left[
\begin{matrix}
f^{\prime}(t_n) & 0\\
q^\prime (t_n) & p^\prime (z_n)
\end{matrix}\right]
$$
Writing $F^j = (F_1^j, F_2^j)$ it follows that
$$
\frac{\partial F_2^j}{\partial z}(t_n,z_n) = \prod_{i=0}^{j-1} p^\prime(z_{n+i}).
$$
We have that the $\omega$-limit set of each $D_n$ contains only finitely many points that lie outside a given neighborhood of the origin. Hence given a $\delta_3>0$, the image
$\pi_2 \circ F^j(D_N)$ lies in the $\frac{\delta_3}{2}$-neighborhood of $0$ for all but the last finitely many $2^\ell N + k \le j < 2^{\ell+1} N$, where the number of $j$'s
for which the image may not lie in the $\frac{\delta_3}{2}$-neighborhood is bounded independently of $\ell$. By taking $\epsilon < \frac{\delta_3}{2}$ we have that $\pi_2\circ
F^j(t_N,z_N) = z_{N+j}$ therefore lies in the $\delta_3$-neighborhood of $0$ for all but finitely many $2^\ell N  + k \le j < 2^{\ell+1} N$. Recall that we have assumed that
$p$ has no critical points besides $x_0$. Hence given $|\hat{\lambda}| < |\lambda|$ we can choose $\delta_3>0$ and $\epsilon>0$ so that there exists an $\delta_4 >0$ for which
$$
\prod_{j= 2^\ell N+1}^{2^{l+1} N } |p^\prime(z_j)| \ge \delta_4 |\hat{\lambda}|^{2^\ell N}.
$$
If we choose $\hat{\lambda}$ and $\tilde{\lambda}$ close enough to $\lambda$ then it follows that
$$
|\frac{\partial F_2^{2^\ell N}}{\partial z}(t_{2^\ell N},z_{2^\ell N})| \ge \delta_4 |\hat{\lambda}|^{2^\ell N} \cdot \delta_2 |\tilde{\lambda}|^{\frac{1-s}{s} 2^\ell N} =
\delta_5 |\lambda|^{2^\ell\frac{N}{2s}},
$$
whose product over all $\ell$ converges to infinity. Therefore $(F^j)$ cannot be a normal family in any neighborhood of $(t_N,z_N)$, which contradicts our initial assumption.
This completes the proof.
\end{proof}

\end{document}